\documentclass[11pt]{amsart}
\usepackage{amsmath}
\usepackage{amssymb}
\usepackage{graphicx}
\usepackage{bm}
\usepackage{hyperref}
\newtheorem{theorem}{Theorem}[section]

\newtheorem{lemma}[theorem]{Lemma}

\theoremstyle{definition}
\newtheorem{definition}[theorem]{Definition}

\newtheorem{remark}[theorem]{Remark}

\numberwithin{equation}{section}
\newtheorem*{acknowledgement*}{Acknowledgments}

\newcommand{\dx}[1]{\mathrm{d}{#1}}
\newcommand{\eqdef}{\stackrel{\mathrm{def}}{=}}

   \def\R{\mathbb{R}}
   
   \def\N{\mathbb{N}}
   \def\Z{\mathbb{Z}}

   \def\wlim{\mathop{w-\rm{lim}}}

   \newcommand{\beq}{\begin{equation}}
   \newcommand{\eeq}{\end{equation}}
   \newcommand{\ba}{\begin{array}}
   \newcommand{\ea}{\end{array}}

\def\i {\mathbf{i}}

\newcommand{\pairing}[2]{\langle #1,#2 \rangle}

\title[Magnetic Schr\"odinger]{Existence in the nonlinear Schr\"odinger equation with bounded magnetic field}
\author[I.Schindler]{Ian Schindler}
\address{Institut de Mathématiques de Toulouse \& TSE, Universit\'{e} Toulouse I Capitole, 1 Esplanade de l'Université, 31080 Toulouse, Cedex 06, France}
\email{ian.schindler@ut-capitole.fr}
\author[C. Tintarev]{Cyril Tintarev}
\email{cyril.tuntarev@protonmail.ch}
\thanks{One of the authors (I.S.) acknowledges funding from ANR under grant ANR-17-EUR-0010 (Investissements d’Avenir program).
Another author (C.T.) thanks CEREMATH at University of Toulouse 1 Capitole for their warm hospitality. He also acknowledges access to the library resources of Uppsala University, magnanimously left to his perusal.}
\date{\today}
\begin{document}

\begin{abstract}
The paper studies existence of ground states for the nonlinear Schrödinger equation
\begin{equation}\label{eq:magschr}
-(\nabla+iA(x))^2u+V(x)u=|u|^{p-1}u	,\quad 2<p<2^*,
\end{equation}
with a general external magnetic field. In particular, no lattice periodicity or symmetry of the magnetic field, or presence of external electric field is required. 
\end{abstract}
\maketitle

%%%%%%%%%%%%%%%%%%%%%%%%%%%%%%%%%%%%%%%%%%%%%%%%%%%%%%%%%%%%%%%%%%%%%%%%%%%%%%%%

\section{\label{intro}Introduction}
We study the existence of solutions for the nonlinear Schrödinger equation with bounded external magnetic field $B$ on $\R^N$. External magnetic field enters Schr\"odinger equation, as well as other equations of quantum mechanics, by addition of a real-valued covector field $A$, called the magnetic potential, to the momentum operator $\frac{\hbar}{i}d$. Magnetic field, which is a measurable quantity, is a differential 2-form $B=d A$, while the magnetic potential is defined up to an arbitrary additive term $d\varphi$, where $\varphi:\R^N\to \R$ is an  arbitrary scalar function. Magnetic momentum 
$\frac{\hbar}{i} d+A$ has the following gauge invariance property:
\begin{equation*}\label{eq:shifts0}
	(\frac{\hbar}{i} d+A)(e^{\i\frac{\varphi}{\hbar}}u)=
	e^{\i\frac{\varphi}{\hbar}}
	(\frac{\hbar}{i}d+A+d\varphi)u.
\end{equation*}
In this paper we follow the convention that normalizes the mass of the particle and sets the Planck constant $\hbar$ to be equal to $1$ (which is always possible by rescaling the time and space variables).

To our best knowledge, the earliest existence result for nonlinear magnetic Schr\"odinger equation 
is the paper by Esteban and Lions \cite{LionsMag} where the magnetic field is assumed to be constant. Their approach was generalized to periodic magnetic field by Arioli and Szulkin \cite{ArioliSzulkin} (see also \cite{SchinTinMag}, \cite{BegoutSchindler}). In addition to that, there is a number of important existence results for quasiclassical solutions, that is, solutions of the nonlinear magnetic Schr\"odinger equation that exist if $\hbar$ is sufficiently small, see \cite{Kurata, Cingolani16, Cingolani17} and references therein, as well as existence for specific magnetic fields such as the Aharonov-Bohm field, \cite{ClappSzulkin}, and studies of properties of solutions to the magnetic Schr\"odinger equation, such as \cite{BonNysSch}.
 
A major technical difficulty in proving existence results for the magnetic Schr\"odinger equation is the lack of compactness of Sobolev embeddings in the whole $\R^N$, and it is overcome by the use of a concentration-compactness argument. For instance, in \cite{ArioliSzulkin}, which generalizes \cite{LionsMag}, one controls the loss of compactness in problems with a periodic magnetic field by means of energy-preserving operators
\begin{equation}
\label{eq:m-shifts}
g_y := \; u\mapsto e^{\i\varphi_y(\cdot)}u(\cdot-y), \quad y\in\Z^N,
\end{equation}
(where $\varphi_y$ is a suitable re-phasing function (see \eqref{eq:phiy} below)),
known as \emph{magnetic shifts} with their inverse
\begin{equation*}
\label{eq:m-invshifts}
g^{-1}_y := \; v\mapsto e^{-\i\varphi_y(\cdot+y)}v(\cdot+y), \quad y\in\Z^N\,.
\end{equation*}
Another concentration mechanism applies to quasiclassical asymptotics and is not considered here. 

The concern of this paper is in finding sufficient existence conditions for solutions to \eqref{eq:magschr} that are not restricted to the periodic magnetic field and, furthermore, do not use strong electric field at infinity to dominate both electric and magnetic field in the whole space, as in the existence condition of \cite{SchinTinMag} 
\begin{equation}
	\label{eq:rostock}
|A(x)|^2 + V(x) < V_\infty\eqdef\lim_{|x|\to\infty}V(x). 
\end{equation}
For the case of periodic magnetic field, a sufficient condition for existence of a ground state, provided by Arioli and Szulkin (generalizing the case of constant magnetic field of \cite{LionsMag}), is $V<V_\infty$, which can be understood as energy penalty by electric field alone when the magnetic field does not change. 
A different condition for existence of a solution (which is not a ground state, and the electric potential is constant) is offered by \cite{DT}:
\begin{equation}
	B_\infty =0, \|B\|_\infty<b_0 ,
\end{equation}
where $b_0$ is a specified positive constant. There is no assumption of periodicity of the magnetic field in \cite{DT}. The lack of periodicity there is handled by a refined concentration-compactness analysis (see Theorem \ref{thm:MagneticPD} and Theorem \ref{thm:energies} below) that uses a suitable non-isometric counterpart of magnetic shifts. While the declared objective of that paper is to describe the concentration mechanism in the general non-periodic case and give a sample existence result. The present paper uses the concentration-compactness approach of \cite{DT} in order to refine the existence condition of \eqref{eq:rostock}. Its existence condition, that uses limits at infinity evaluated via lattice shifts,
 \begin{equation}
	\label{eq:not rostock}
	|A(x)|^2 + V(x) <  |A_\infty (x)|^2+V_\infty(x),
	\end{equation}
does not require the magnetic field to be dominated by the electric field at infinity. Remarkably, however, this condition is also accompanied with a requirement of sufficiently small $B$ like in \cite{DT} although this time it arises to assure that the problem at infinity has a nonvanishing ground state. 

The paper is organized as follows. In Section~2 we recall basic properties of magnetic Sobolev spaces associated with respective energy functionals, define (non-isometric) magnetic shift operators as well as asymptotic values of magnetic field and energy at infinity. In Section~3 we quote the structural result from \cite{DT} about asymptotic behavior of sequences in magnetic Sobolev spaces. This result, usually called profile decomposition, was first proved, for general sequences in Sobolev spaces, by Solimini \cite{Sol}, and then extended, to Hilbert spaces and to Banach spaces in \cite{SchinTin} and in \cite{ST1} (with further details in \cite{DSU}), respectively. The profile decomposition of \cite{DT} that we use here is set in a Frechet space and does not follow from prior functional-analytic results, it is, however, close in spirit to profile decomposition for Sobolev spaces on manifolds, \cite{SkrTi}. 

The main results of this paper are Theorem~\ref{thm:1}, Theorem~\ref{thm:2} and Theorem~\ref{thm:PD*}. 

Two first two theorems state existence of minimizers in the constraint problem \eqref{eq:kAV} under a penalty condition. The third one deals with a model minimization problem \eqref{eq:k*} involving the Aharonov-Bohm magnetic potential, a singular electric potential, and critical Sobolev nonlinearity.

The minimizers, also called ground states, satisfy the Euler-Lagrange equation 
\[
-(\nabla+iA(x))^2u+V(x)u=\lambda|u|^{p-1}u	
\]
with some $\lambda>0$. Since the right and the left hand sides of the equation have different homogeneity, a suitable scalar multiple of $u$ satisfies \eqref{eq:magschr}.

In Appendix we present profile decompositions for sequences in Sobolev spaces in relation to critical Sobolev embeddings.

\section{\label{notation} Preliminaries: magnetic Sobolev space, magnetic shifts, energy at infinity}
Throughout the paper the set $\N$ of natural numbers will be understood to include zero.
To an extent that it causes no ambiguity, we will treat the magnetic potential $A$ either as a linear form on $R^N$ or as its coordinate representation as a vector field on $\R^N$. We define 
\begin{equation}\label{eq:nablaA}
\nabla_A (u)\eqdef  (\nabla+ \i A)u,
\end{equation}
and the scalar product
\begin{equation*}\label{eq:scalar}
\pairing{u}{v}_A\eqdef\int_{\R^N} \nabla_A u(x) \overline{\nabla_A v(x)}\dx{x}\,,
\end{equation*}
(the symbol of the scalar product on $\R^N$ will be omitted throughout the paper). 
Then, we introduce the (Dirichlet-type) energy functional
\begin{equation}\label{eq:hamiltonian}
E_A(u)\eqdef \int_{\R^N}|\nabla_A u(x)|^2\dx x.
%=\int_{\R^N}|\nabla u(x)+ \i A(x) u(x)|^2\dx x.
\end{equation}
Elementary calculations using the Cauchy inequality yield the following well known pointwise estimates:

\begin{equation}\label{eq:H1ABound}
|\nabla_A u(x)|^2\geq \frac{1}{2}|\nabla u(x)|^2-8|A(x)|^2 \, |u|^2\,,x \in \R^N,\\
\end{equation}
or, conversely,
\begin{equation}\label{eq:H1Bound}
|\nabla_A u(x)|^2\leq 2|\nabla u(x)|^2+16|A(x)|^2 \, |u|^2\,,x \in \R^N.\\
\end{equation}
In this section we assume 
\begin{equation}
	\label{LN}
\begin{cases}
	A \in L^N_{\mathrm loc}(\R^N) 	& N \ge 3,\\
	A \in L^{2 + \epsilon}_{\mathrm loc}(\R^2) 	& N = 2, \epsilon >0,
\end{cases}
\end{equation}
or that $A$ is the Aharonov-Bohm potential $A(x)=\lambda\frac{(x_1,-x_2,0,\dots,0)}{x_1^2+x_2^2}$
Further on, $A$ will be restricted to a smaller class. 
When dealing with a magnetic potential $A$ we shall define $\dot H_A^{1,2}(\R^N)$ as the completion of $C_0^\infty$ with respect to the norm $\|u\|_A=(E_A(u))^\frac12$
(note that $\dot H^{1,2}_A(\R^N)$ is a space of measurable functions whenever $N>2$ or $N=2$ and $dA\not\equiv 0$, see e.g.\ \cite{Enstedt}), and the space 
$H_A^{1,2}(\R^N)$ as the intersection $\dot H_A^{1,2}(\R^N)\cap L^2(\R^N)$ (equipped with the standard intersection norm). From \eqref{LN} it easily follows that $H_A^{1,2}(\R^N)$
is continuously embedded in $H^{1,2}_{\mathrm loc}(\R^N)$ and that Frechet spaces $H_{A\,\mathrm loc}^{1,2}(\R^N)$ qnd $H^{1,2}_{\mathrm loc}(\R^N)$ coincide (for details see, for example, \cite{DT}).
 Note that $|u|^2$ in this model equals the probability density of a particle in space which doesn't change if $u$ is multiplied by $e^{\i\varphi}$ (on the other hand the magnetic potential $A$ and $A+\nabla\varphi$ give rise to the same magnetic field $B=dA$).

Note also that, when $A\equiv 0$, we obtain the corresponding usual Sobolev spaces $\dot H^{1,2}(\R^N)$ and $H^{1,2}(\R^N)$ respectively.
Setting, as usual, $2^*=\begin{cases}
\frac{2N}{N-2}, &N>2\\
\infty, &N=2 \end{cases}$
the critical Sobolev exponent, we have the following continuous embeddings
$H^{1,2}(\R^N)\hookrightarrow L^p(\R^N)$ for every $p\in[2,2^*)$,
 and $H^{1,2}(\R^N) \hookrightarrow \dot H^{1,2}(\R^N)\hookrightarrow  L^{2^*}(\R^N)$ when $N>2$. Furthermore, from the well-known {\em{diamagnetic inequality}} 
\begin{equation}\label{eq:diamag}
|\nabla_A u(x)|\ge \left|\nabla|u(x)|\right|,
\end{equation}
we deduce that
\begin{equation*}
\label{eq:ineq}
E_A(u)\ge E_0(|u|)\ge C\|u\|_{2^*}^2, \quad \mbox{ if } N>2,
\end{equation*}
and, as a consequence, the continuous embeddings
\begin{equation*}
\label{eq:emb}
\dot H_A^{1,2}(\R^N)\hookrightarrow L^{2^*}(\R^N)\quad \mbox{ and }\quad H_A^{1,2}(\R^N)\hookrightarrow  L^{p}(\R^N) \;\forall p\in[2,2^*),
\end{equation*}
(the latter one extends by analogous argument to the case $N=2$).

Let $N\ge 2$ and let $\Lambda_1$ and $\Lambda_2$ denote the linear spaces, respectively, of 1-forms on $\R^N$ and of antisymmetric 2-forms on $\R^N\times\R^N$.
Note that when $B$ is a constant magnetic field, its magnetic potential $A$ equals, up to a gradient of an arbitrary function, $\frac12 Bx$. Its linear growth suggests that, generally, one shouold not try to define, for a given function $u\in H^{1,2}_A(\R^N)$, the limit of the magnetic energy functional \eqref{eq:hamiltonian} when the function $u$ is shifted at infinity like in the non-magnetic case (e.g.\ \cite{Lions1}), by taking $\lim_{|y|\to\infty} E_0(u(\cdot+y))$. The appropriate construction takes advantage, instead, of the gauge invariance of the magnetic energy.
In the case of lattice-periodic $B$, Arioli and Szulkin \cite{ArioliSzulkin} derived energy functionals at infinity using "magnetic shifts" 
\begin{equation}
	g_y := u\mapsto  e^{\i \varphi_y(\cdot)} u(\cdot-y) \quad \forall u\in C_0^1(\R^N),\;y\in\Z^N,
\end{equation}
where the function $\varphi_y$ is defined by the property 
\begin{equation}
	\label{eq:phiy}
	A(\cdot -y)=A(\cdot)+\nabla\varphi_y(\cdot).
\end{equation}
(Indeed, this relation is equivalent to lattice-periodicity of the field $dA(\cdot-y)=dA$.)

In \cite{DT} analogous magnetic shifts were introduced without assumption of perioidicity of the magnetic field (the case $N=2$ was not included, but the original proof remains valid in this casse as well):

\begin{lemma}[\cite{DT}]
	\label{lem:phase2}
	Let $B=d A$ be a bounded magnetic field with $A\in C^1_\mathrm{loc}(\R^N,\Lambda_1)$ and $A(0)=0$. For every  $y\in\R^N$ there exists a function $\varphi_y\in C^1_{\mathrm loc}$, such that $\varphi_0=0$ and
	\begin{equation}
		\label{eq:Aineq}
		|A+\nabla\varphi_y|\le \|B\|_\infty\,|x-y|,\quad  x\in\R^N,
		\end{equation}
	where 
	\begin{equation}
		\label{eq:Binfty}
		\|B\|_\infty^2\eqdef \sum_{n=1}^N\sum_{m<n}\|\partial_n A_m-\partial_m A_n\|_\infty^2,
	\end{equation}
\end{lemma}

The function $\varphi$ given by Lemma~\ref{lem:phase2} is not uniquely defined. A construction of such function is given by \cite{DT} in Lemma~3.1.
Defining the corrected potential $A_y$ by 
\begin{equation}
	\label{eq:correctedA}
	A_y\eqdef A+\nabla\varphi_y,
\end{equation}
note that
by \eqref{eq:Aineq} the value of 
$A_y$  at $y$ is
is zero.

\begin{definition}[\cite{DT}]
\label{def:g}
Let $B=d A$ be a magnetic field with $A\in C^1_{\mathrm{loc}}(\R^N, \Lambda_1)$. 
Any map 
\begin{equation}
\label{eq:g}
g_y := u\mapsto  e^{\i \varphi_y(\cdot)} u(\cdot-y), \quad u\in C_0^1(\R^N),\;y\in\R^N,
\end{equation}
with the $C^1(\R^N)$ function $\varphi_y$ as in Lemma \ref{lem:phase2} is called \emph{a magnetic shift} (relative to the magnetic field $B$) determined by the vector $y$.
\end{definition}

Elementary calculations  using 
\eqref{eq:correctedA} (which hold for any continuous real-valued function in the place of $\varphi_y$) show that 
\begin{equation}\label{eq:commutation}
\nabla_A(g_y u)= g_y(\nabla_{A_y(\cdot +y)}u)\quad
\mbox{ and }\quad
(g_y)^{-1}(\nabla_A u)= \nabla_{A_y(\cdot +y)}((g_y)^{-1}u).
\end{equation}
Furthermore, we get
\begin{equation*}\label{eq:scalargauge}
\pairing{g_y u}{g_y v}_A= \pairing{u}{v}_{A_y(\cdot +y)}\quad
\mbox{ and }\quad
\pairing{(g_y)^{-1}u}{(g_y)^{-1}v}_{A_y(\cdot +y)}=\pairing{u}{v}_A,
\end{equation*}
and, in particular, that 
\begin{equation}\label{eq:EAphasing}
E_A(g_y u) =
E_{A_y(\cdot+y)}(u)
\quad
\mbox{ and }\quad
E_A(u)=E_{A_y(\cdot +y)}((g_y)^{-1}u).\\
\end{equation}

In what follows $\dot C^1(\R^N)$ denotes the space of functions with uniformly bounded continuous derivatives and  $\dot C^{0,1}(\R^N)$ denotes the space of functions satisfying $|f(x)-f(y)|\le C|x-y|$ for all $x,y\in\R^N$. In particular,  $\dot C^1(\R^N)$ contains all linear functions on $\R^N$.
\begin{remark}\label{rem:Lipschitz}
	Let $Y=(y_k)_{k\in\N}$ be a divergent sequence in $R^N$. If $A\in \dot C ^1(\R^N,\Lambda_1)$, then applying Arzel\`a-Ascoli theorem 
	to the sequence $(A_{y_k}(\cdot + y_k))_{k\in\N}$,
	we get a renamed subsequence such that 
	$(A_{y_k}(\cdot+y_k))_{k\in\N}$ converges, uniformly on bounded sets to some Lipschitz  function $A^{(Y)}_\infty\in \dot C^{0,1}(\R^N,\Lambda_1)$, and and $(dA_{y_k}(\cdot+y_k))_{k\in\N}$ converges almost everywhere with a constand bound to a bounded function $B^{(Y)}_\infty=dA^{(Y)}_\infty\in L^\infty(\R^N,\Lambda_2)$. 
\end{remark}
\begin{definition}
\label{def:AY}
Let $Y= (y_k)_{k\in\N}$ be a divergent sequence in $\R^N$,
% such that $|y_k|\to\infty$, 
and consider the renamed subsequence $(y_k)_{k\in\N}$ and the associated magnetic potential $A^{(Y)}_\infty\in C^{0,1}_\mathrm{loc}(\R^N,\Lambda_1)$ described in Remark~\ref{rem:Lipschitz}. Then, the corresponding functional $E_{A^{(Y)}_\infty}$ (see \eqref{eq:hamiltonian}) will be called \emph{the energy at infinity relative to the renamed subsequence $Y=(y_k)_{k\in\N}$}, corresponding to the \emph{magnetic field at infinity} $B^{(Y)}_\infty= dA^{(Y)}_\infty$. 
\end{definition}

\section{Profile decomposition}\label{PD}
In this section we quote  from \cite{DT} the results on structure of bounded sequences in $H_A^{1,2}(\R^N)$. While the original result is stated for $N\ge 3$, the original argument extends to the case $N=2$ with no modifications. 
\begin{definition}
	A set $\Xi\subset \R^N$ is called a discretization of $\R^N$ if $0\in \Xi$,
	\[
	\inf_{x,y\in \Xi, x\neq y} |x-y|>0,
	\]
	and if, for some $\rho>0$, $\{B_\rho(x)\}_{x\in\Xi}$ is a covering of $\R^N$ of uniformly finite multiplicity.
	
	%\[
	%\{B_\rho(x)\}_{x\in\Xi}\mbox{ is a covering of $\R^N$ of uniformly finite multiplicity}.
	%\]
\end{definition}
Note that if $\Xi$  is a discretization of $\R^N$ then, for any $R>\rho$, the covering $\{B_R(x)\}_{x\in\Xi}$ is still of uniformly finite multiplicity. A trivial example for a discretization of $\R^N$ is $\Z^N$.
%

%We will use the notation:
%$\N_0\eqdef \N\setminus\{0\}$, and, for simplicity,
When we shall deal with sequences 
$Y^{(n)}\eqdef (y_k^{(n)})_{k\in\N}$ in $\R^N$ depending on a parameter $n\in \N$, we will abbreviate the notation for the magnetic potentials at infinity $A^{(Y^{(n)})}_\infty$ as $A^{(n)}_\infty$ (and the corresponding magnetic field as $B^{(n)}_\infty$).
Moreover, for any $n$ and for any $k$, $g_{y^{(n)}_k}$ will represent the magnetic shift defined by \eqref{eq:g} with $y=y^{(n)}_k$, and we will use the abbreviated notation
\begin{equation}\label{eq:short}
g^{(n)}_k\eqdef g_{y^{(n)}_k}\quad \mbox{ as well as }
\quad A^{(n)}_k \eqdef A_{y^{(n)}_k}(\cdot +y^{(n)}_k)\,.
\end{equation}
In what follows we use weak convergence in $H^{1,2}_{\mathrm{loc}}(\R^N)$, which is a weak convergence in a Fr\'echet space. In particular weakly converging sequences will converge in the sense of distributions, strongly in $L^p(\Omega)$ whenever $\Omega\subset \R^N$ is a bounded measurable set and $p\in(2,2^*)$, as well as almost everywhere. 
\begin{theorem}
	\label{thm:MagneticPD} Let  $\Xi\ni 0$ be a discretization of $\R^N$.
	Let $A\in \dot C^1(\R^N,\Lambda_1)$, $N\ge 2$, $A(0)=0$,
	and let	$(u_{k})_{k\in\N}$ be a bounded sequence in $H_A^{1,2}(\R^N)$.
	Then, there 
	exist $v^{(n)}  \in
	H^{1,2}_\mathrm{loc}(\R^N)$ (with $v^{(0)}=\wlim_{k\to \infty} u_k\in H_A^{1,2}(\R^N)$), $Y^{(n)}:=(y_{k}^{(n)})_{k\in\N}\subset \Xi$, $n\in \N$,
	with $y_{k} ^{(0)} =0$, 
	%with $n\in\N_0$,
	such that, on a renamed subsequence, 
\begin{eqnarray}
\label{mag-shifts}  
	%\exp({i\varphi_{y_k^{(n)}}(\cdot+y_k^{(n)})})u_k(\cdot+y_{k} ^{(n)})\rightharpoonup v^{(n)},
&& (g^{(n)}_k)^{-1} u_k \rightharpoonup v^{(n)}\quad \mbox{ in } H^{1,2}_{\mathrm{loc}}(\R^N),\\
&& \label{separates}
|y_{k} ^{(n)}-y_{k} ^{(m)}|\to\infty  \mbox{ for } n \neq m,\\
&& \label{BBasymptotics-mag}
	r_k\eqdef u_{k} - \sum_{n=0}^\infty
		g^{(n)}_kv^{(n)} \to 0\, \mbox{ in }
		L^p(\R^N), \,\forall p\in (2,2^*),
			\end{eqnarray}
and the series $\sum_{n=0}^\infty
		|v^{(n)}(\cdot-y_{k}^{(n)})|$ and the series in (\ref{BBasymptotics-mag}) converge
	 unconditionally (with respect to $n$) and uniformly in $k$, in 
$H^{1,2}(\R^N)$ and in $H^{1,2}_\mathrm{_{loc}}(\R^N)$ respectively.
Moreover, for any $n$,
\begin{equation}\label{eq:gsaturation}
g^{-1}_{y_k}u_k\rightharpoonup 0  \mbox{ in } H^{1,2}_\mathrm{_{loc}}(\R^N)\quad \mbox{ for all } \; (y_k)_{k\in\N} \mbox{ in } \R^N\mbox{ s.t. }|y_k-y^{(n)}_k|\to+\infty.
\end{equation}
\end{theorem}
\begin{theorem}\label{thm:energies} Assume conditions of Theorem~\ref{thm:MagneticPD}. A subsequence $(u_{k})_{k\in\N}$ in $H_A^{1,2}(\R^N)$ provided by Theorem~\ref{thm:MagneticPD} satisfies
	\begin{equation}\label{eq:IBL}
	\int_{\R^N}|u_k|^p\dx{x}\longrightarrow \sum_{n=0}^\infty\int_{\R^N}|v^{(n)}|^p\dx{x}\quad \mbox{ as } k\to +\infty\;\forall p\in[2,2^*),
	\end{equation}
		\begin{equation}\label{eq:IBL2}
 \sum_{n=0}^\infty\int_{\R^N}|v^{(n)}|^2\dx{x}\le \liminf_{k\to\infty} \int_{\R^N}|u_k|^2\dx{x},
	\end{equation}
and	furthermore, 
	\begin{equation} 
	\label{norms-mag} 
	\sum_{n =0}^\infty E_{A^{(n)}_\infty}(v ^{(n)})\le \liminf_{k\to \infty} E_A(u_k),
	\end{equation}
where $A^{(0)}_\infty=A$, and for $n\neq 0$, each $A^{(n)}_\infty$ is the magnetic potential at infinity relative to the sequence $Y^{(n)}=(y_k^{(n)})_{k\in\N}$ as in Definition~\ref{def:AY} and $B^{(n)}_\infty=d A^{(n)}_\infty$ is the related magnetic field.
\end{theorem}

\section{Applications}
In this section make the following assumption:
\begin{equation}
	\label{AV}
A\in \dot C^1(\R^N,\Lambda_1),\; V\in C(\R^N), \inf_{x \in \R^N} V(x) > 0, 
\end{equation}
and consider
the following functional on $H^{1,2}_A(\R^N)$.
\begin{equation}\label{eq:JAV}
J_{A,V}(u)\eqdef\int_{\R^N}\left(|(\nabla_A u(x)|^2+V(x)|u(x)|^2\right)\dx{x}.
\end{equation}
For a fixed $p\in (2,2^*)$ we consider a minimization problem 
\begin{equation}\label{eq:kAV}
	\kappa_{A,V}\eqdef\inf_{ u \in H^{1,2}_A(\R^N)\int_{\R^N}|u(x)|^p\dx{x}=1}J_{A,V}(u).
\end{equation}
A minimizer for this problem is also called \textit{ground state}.
 Given a sequence $Y = (y_k)_{k\in \N}\subset{\R^N}$, let 
 $A_\infty^{(Y)}$ be as in Remark~\ref{rem:Lipschitz}, let
 \begin{equation}
	V_\infty^{(Y)}(x)\eqdef \liminf V(x+y_k),
 \end{equation} 
 and let $J_\infty^{(Y)}$ be the functional \eqref{eq:JAV} with $A$ and $V$ repaced, respectively, by $A_\infty^{(Y)}$ and $V_\infty^{(Y)}$. Furthermore, define $\kappa_\infty^{(Y)}$ as the constant $\kappa_{A,V}$ with $J_{A,V}$ replaced by $J_\infty^{(Y)}$. 	\begin{lemma}\label{lemcc}
 	Assume \eqref{AV}.  Let $\Xi$ be a discretization of $\R^N$. If for every sequence $Y = (y_k)_{k\in \N}\subset{\Xi}$, 
 	\begin{equation}\label{penalty}
 		\kappa_\infty^{(Y)}>\kappa_{A,V},
 	\end{equation}
 	then the minimum in \eqref{eq:kAV} is attained and any minimizing sequence has a subsequence convergent to a minimizer.
 \end{lemma}
 \begin{proof}
 	Let $(u_k)$ be a minimization sequence for \eqref{eq:kAV}. Applying Theorem \ref{thm:MagneticPD}, and using  the notation $V_\infty^{(n)}$ for  $V_\infty^{Y^(n)}$ (here and further on in the similar context),
 	we easily get
 	\begin{equation}\label{eq:IBL3}
 		\liminf_{k\to\infty} \int_{\R^N}V|u_k|^2\dx{x}\ge\sum_{n=0}^\infty\int_{\R^N}V_\infty^{(n)}|v^{(n)}|^2\dx{x},
 	\end{equation}
 	which, combined with \eqref{norms-mag}, yields
 	\begin{equation}\label{eq:nrj}
 		\kappa_{A,V}=\liminf_{k\to\infty} J_{A,V}(u_k)\ge
 		\sum_{n=0}^\infty J_\infty^{(n)}(v^{(n)}).
 	\end{equation}
 	At the same time, setting $t_n=\int_{\R^N}|u|^p\dx{x}$, we have from \eqref{eq:IBL} 
 	\[
 	\sum_{n=0}^\infty t_n=1. 
 	\]
 	Then from \eqref{eq:nrj} follows
 	\[
 	\kappa_{A,V}\ge \sum_{n=0}^\infty \kappa_\infty^{(n)}t^\frac2p,
 	\]
 	which in view of \eqref{penalty} can hold only if $t_0=1$ and for all $n\ge 1$, $t_n=0$. This impies that $J_{A,V}(v^{(0)})=\kappa_{A,V}$
 	and therefore $u_k\to v^{(0)}$ in $H_A(\R^N)$ and $v^{(0)}$ is a minimizer.
 \end{proof}
 This lemma allows us to reduce the proof of existence of ground state below to verification of \eqref{penalty}.
 
 \begin{theorem}\label{thm:2}
 	Assume \eqref{AV}.  Let $\Xi$ be a discretization of $\R^N$. If for every divergent sequence $Y = (y_k)_{k\in \N}\subset{\Xi}$,
 	\begin{description}
 		\item[(A)] the problem \eqref{eq:kAV} with $A$, $V$ replaced, respectively with $A_\infty^{(Y)}$ $V_\infty^{(Y)}$ has a ground state
 		$v$, and
 		\item[(B)] $|A_\infty^{(Y)}(x)|^2 +2A_\infty^{(Y)}(x)\cdot \mathrm{Im}\frac{\nabla v(x)}{v(x)} +V_\infty^{(Y)}(x)<$ \\$< 
 		|A(x)|^2 +2A(x)\cdot \mathrm{Im}\frac{\nabla v(x)}{v(x)} +V(x)$, whenever $v(x) \neq 0$. 
 	\end{description}
 	then the minimum in \eqref{eq:kAV} is attained and any minimizing sequence has a subsequence convergent to a minimizer in $H^{1,2}_A(\R^N)$.
 \end{theorem}
 \begin{proof}
 	Let us use the following identity which holds for all $x$ such that $u(x) \neq 0$:
 	\begin{align*}
 		|\nabla_A u|^2 &= |u|^2 \left |\frac{\nabla u}u +iA \right |^2
 		\\
 		&= |u|^2 \left |\frac{\nabla |u|}{|u|} + i\mathrm{Im}\frac{\nabla u}u+iA \right |^2\\
 		&=\left|\nabla| u|\right|^2 +  \left | \mathrm{Im}\frac{\nabla u}u+A \right |^2|u|^2
 	\end{align*}
 	When $u(x) = 0$, $	|\nabla_A u(x)| = |\nabla u(x)|$. 
 	Using condition \textbf{(B)} we have 
 	\begin{align*}
 		\nonumber \kappa_\infty^{(Y)} & = J_\infty^{(Y)}(v)\\
 		\nonumber & = \int_{\R^N} \left (|\nabla |v||^2 + \left|A_\infty^{(Y)}+\mathrm{Im}\frac{\nabla v}{v}\right |^2|v|^2 + V_\infty^{(Y)}|v|^2\right )\dx{x} \\
 		&= \int_{v \neq 0} \left (|\nabla |v||^2 + \left|A_\infty^{(Y)}+\mathrm{Im}\frac{\nabla v}{v}\right |^2|v|^2 + V_\infty^{(Y)}|v|^2\right)\dx{x} +\int_{v = 0} |\nabla v|^2 \dx{x} \\
 		& \ge \int_{v \neq0} \left(|\nabla |v||^2 +  \left|A+\mathrm{Im}\frac{\nabla v}{v}\right|^2|v|^2 \dx{x}+ V|v|^2\right)\dx{x} +\int_{v = 0} |\nabla v|^2\dx{x} \\
 		&= J_{A,V}(v) \ge \kappa_{A,V}. 
 	\end{align*}
 	The assertion follows from Lemma~\ref{lemcc}.
 \end{proof}

\begin{remark}
	Conditions \eqref{AV} can be relaxed, taking into account that asymptotic values of the electric and the magnetic potentials, as well as concentration profiles $v^{(n)}$ in Theorem~\ref{thm:MagneticPD} are independent of the values of $A$ and $V$ in a fixed bounded set. One can assume instead that in a ball of radius $R$, centered at the origin, $A \in L^N(B_R(0))$ and $V \in L^{N/2}(B_R(0))$ when $N \ge 3$.  When $N=2$ we may require $A \in L^{2+\epsilon}(B_R(0))$ and $V \in L^{1+\epsilon}(B_R(0))$ for some $\epsilon > 0$. More generally it suffices to have $|A|^2$ and $V$ in the suitable class of continuous multiplier operators, for example the Hardy potential $V(x)=\frac{a}{|x|^2}$ and, for $N=2$, the Aharonov-Bohm potential  $A(x)=a\frac{(x_1,-x_2}{x_1^2+x_2^2}$, $a \in \R$.
\end{remark}
\begin{remark}
The theorem above takes a simpler form if we assume that the ground state for every problem at infinity is positive. This is indeed the case if the magnetic field at infinty is constant ($B(x)=B\in \Lambda_2$, $A=\frac12Bx$) and is sufficiently small, see e.g. \cite[Proposition 5.7]{BonNysSch}.
\end{remark}

 	\begin{theorem}\label{thm:1}
 	Assume \eqref{AV}.  Let $\Xi$ be a discretization of $\R^N$. If for every divergent sequence $Y = (y_k)_{k\in \N}\subset{\Xi}$,
 	\begin{description}
 		\item[(i)] 	the problem \eqref{eq:kAV} with $A$, $V$ replaced, respectively, by $A_\infty^{(Y)}$ and $V_\infty^{(Y)}$ has a positive ground state, and condition \textbf{(B)} holds, which in this case reads as
 		\item[(ii)] $|A(x)|^2 + V(x) < | A_\infty^{(Y)}(x)|^2+ V_\infty^{(Y)}(x)$, $x\in\R^N$,
 	\end{description}
 	then the minimum in \eqref{eq:kAV} is attained and any minimizing sequence has a subsequence convergent to a minimizer.
 \end{theorem}

\begin{remark} 
	What is the point of conditioning the existence of a ground state by the existence of other ground states?  The answer lies in the fact that existence of a ground states is already known for  some magnetic fields that may appear as limits at infinity of other magnetic fields.  In particular the existence of a ground state is known in the case of periodic magnetic fields \cite{ArioliSzulkin} (which includes constant magnetic fields \cite{LionsMag}), or under the assumption \eqref{eq:rostock} in \cite{SchinTinMag}. 
\end{remark}

\section{Critical exponent problem}
Assume that $N\ge 3$ and let
\begin{eqnarray}
\label{eq:Ahab}
	J(u)\eqdef\int_{\R ^{N}}\left(
	\bigl| \nabla_Au \bigr| ^{2}-\frac{\mu}{x_1^2}\bigl|u \bigr| ^{2}\right)
	\dx{x}\\
\nonumber	\text{ with } A(x)=\lambda\frac{(x_1,-x_2,0,\dots,0)}{x_1^2+x_2^2}, \quad \lambda,\mu\in\R
\end{eqnarray}

\begin{theorem} Assume that $\lambda^2\le\mu<\frac14$. Then the  minimum in the problem
\begin{equation}\label{eq:k*}
	\kappa\eqdef \inf_{u \in \dot H^{1,2}_A(\R^N):\int_{\R^N}|u(x)|^{p^*}\dx{x}=1}J(u)
\end{equation}
is attained.
\end{theorem}
\begin{proof}
By the diamagnetic and the Hardy inequalities $J(u)$ defines an equivalent  $\dot H^{1,2}_A(\R^N)$-norm, so that $\kappa>0$. Let $(u_k)$ be a minimizing sequence for \eqref{eq:k*} and consider the subsequence provided by Theorem~\ref{thm:PD*}.
Consider the action of scaling on the electric potential V(x)=-$\frac{\mu}{x_1^2}$:
\begin{equation}
	\int_{\R ^{N}}V(x)|gu|^2\dx{x}=\int_{\R ^{N}}2^{-2j}V(2^{j}\cdot+y)|u|^2\dx{x},
\end{equation}
which in our particular case yields the value of the rescaled electric potential as $V(\cdot+2^jy)$. Similar calculation gives our rescaled potential equal $A(\cdot+2^jy)$. Calculation of their respective limit values yields either $0$ or, respectively $A,V$ possibly up to a fixed scaling, which can be eliminated by redefining the respective concentration profile $w^{(n)}$ by an inverse rescaling.
Let us write \eqref{eq:2*} in the form 
$
\sum_{n\in \N }t_n=1, t_n=\int_{\R ^{N}}
\bigl|w^{(n)} \bigr| ^{2^*}
\dx{x} .$
Then from \eqref{eq:normSob*} follows
\begin{equation}\label{eq:energy*}
\sum_{n\in\N} \kappa_nt_n^{2/2^*}\le \kappa,
\end{equation}
where $\kappa_n$ has one of two values, $\kappa$ or the Sobolev constant $S_N$, which under the assumption $\lambda^2\le\mu$ is greater than $\kappa$. Indeed, with $v$ denoting the Talenti minimizer,
\[
S_N=\int_{\R^N}|\nabla v|^2\dx{x} = \int_{\R^N}\left(|\nabla_A v|^2-|A(x)|^2v^2\right)\dx{x}>J(v)\ge\kappa.
\]

Consequently, \eqref{eq:energy*} can hold only if all $\kappa_n$ but one are zero, and the remaining $\kappa_{n_0}$ equals $\kappa$. Then the corresponding $w^{(n_0)}$ (after the translation that makes $A^{(n_0)}=A$ and $V^{(n_0)}=V$) is a minimizer for \eqref{eq:k*}.
\end{proof}

\section*{Appendix}
In homogeneous Sobolev spaces equipped with the group of shifts and dilations, one has the following profile decomposition of Solimini which we quote in a slightly refined version of \cite[Theorem 4.6.4]{ccbook}.

\begin{theorem}
	[Sergio
	Solimini,  \cite{Sol}]\label{thm:PDsob}%
	Let $(u_{k})$ be a bounded sequence in $\dot{H}^{m,p}(\R ^{N})$,
	$m\in \N $, $1<p<N/m$. Then it has a renamed subsequence and there exist
 sequences of isometries on $\dot{H}^{m,p}(\R ^{N})$, $(g_{k}^{(n)})_{k\in \N })_{n\in \N }$,
and functions	$w^{(n)}\in \dot{H}^{m,p}(\R ^{N})$, such that $g_{k}^{(n)}u\eqdef 2^{
		\frac{N-mp}{p}j_{k}^{(n)}}u(2^{j_{k}^{(n)}}(\cdot -y_{k}^{(n)}))$,
	$j_{k}^{(n)}\in \Z $, $y_{k}^{(n)}\in \R ^{N}$, $n\in\N$, with
	%
	%e4.20 #&#
	\begin{equation}
	%	\label{eq:profileHm}
	 j_{k}^{(0)}=0, y_{k}^{(0)}=0,		\bigl| j_{k}^{(n)}-j_{k}^{(m)} \bigr
		| +\bigl(2^{j_{k}^{(n)}}+2^{j
			_{k}^{(m)}}\bigr) \bigl|
		y_{k}^{(n)}-y_{k}^{(m)} \bigr| \to
		\infty ,m\neq n,
	\end{equation}
	such that $[g_{k}^{(n)}]^{-1}u_{k}\rightharpoonup w^{(n)}$ in
	$\dot{H}^{m,p}(\R ^{N})$,
	%
	%e4.21 #&#
	\begin{equation}
		%\label{eq:B-Basymp*}
		 u_{k}-\sum_{n\in \N }
		g_{k}^{(n)}w^{(n)}\to 0 \text{in } L^{p_{m}
			^{*}}
		\bigl(\R ^{N}\bigr),
	\end{equation}
	the series $\sum_{n\in \N } g_{k}^{(n)}w^{(n)}$ converges in
	$\dot{H}^{m,p}(\R ^{N})$ unconditionally and uniformly in $k$, and
	%
	%e4.22 #&#
	\begin{equation}
		%\label{eq:normSob*} 
		\sum_{n\in \N }\int_{\R ^{N}}
		\bigl| \nabla ^{m}w^{(n)} \bigr| ^{p}
		\dx{x} \le \liminf \int_{\R ^{N}} \bigl| \nabla
		^{m}u_{k} \bigr| ^{p}\dx{x} .
	\end{equation}
\end{theorem}

This theorem cannot be directly applied to sequences bounded in $\dot H^{1,2}_A(\R^N)$ with the magnetic potential as in \eqref{eq:Ahab}. However its proof can be trivially modified for this space, given that  $\dot H^{1,2}_A(\R^N)$ is continuously embedded into $H^{1,2}_\mathrm{loc}(\R^N)$ and into $L^{p^*}(\R^N)$. The modifications are analogous to the argument used in the proof of Theorem~\ref{thm:MagneticPD} in the subcritical case and are omitted. We have:

\begin{theorem}\label{thm:PD*}%
	Let $(u_{k})$ be a bounded sequence in $\dot H^{1,2}_A(\R^N)$, $N\ge 3$. Then it has a renamed subsequence and there exist
functions	$w^{(n)}\in {H}_\mathrm{loc}^{1,2}(\R ^{N})$ and sequences of bounded operators in ${H}_\mathrm{loc}^{1,2}(\R ^{N})$, $(g_{k}^{(n)})_{k\in \N }$,
such that $g_{k}^{(n)}u\eqdef 2^{
		\frac{N-2}{2}j_{k}^{(n)}}u(2^{j_{k}^{(n)}}(\cdot -y_{k}^{(n)}))$,
	$j_{k}^{(n)}\in \Z $, $y_{k}^{(n)}\in \R ^{N}$, $n\in\N$, with
	%
	%e4.20 #&#
	\begin{equation}
		\label{eq:profileHm} j_{k}^{(0)}=0, y_{k}^{(0)}=0,		\bigl| j_{k}^{(n)}-j_{k}^{(m)} \bigr
		| +\bigl(2^{j_{k}^{(n)}}+2^{j
			_{k}^{(m)}}\bigr) \bigl|
		y_{k}^{(n)}-y_{k}^{(m)} \bigr| \to
		\infty ,m\neq n,
	\end{equation}
	such that $[g_{k}^{(n)}]^{-1}u_{k}\rightharpoonup w^{(n)}$ in ${H}^{1,2}_\mathrm{loc}(\R ^{N})$,
	%
	%e4.21 #&#
	\begin{equation}
		\label{eq:B-Basymp*} u_{k}-\sum_{n\in \N }
		g_{k}^{(n)}w^{(n)}\to 0 \text{ in } L^{p
			^{*}}
		\bigl(\R ^{N}\bigr),
	\end{equation}
	the series $\sum_{n\in \N } g_{k}^{(n)}w^{(n)}$ converges in
${H}^{1,2}_\mathrm{loc}(\R ^{N})$ unconditionally and uniformly in $k$, and
	%
	%e4.22 #&#
	\begin{equation}
		\label{eq:normSob*} \sum_{n\in \N }\int_{\R ^{N}}
		\bigl| \nabla_{A^{(n)}}w^{(n)} \bigr| ^{2}
		\dx{x} \le \liminf \int_{\R ^{N}} \bigl| \nabla_A u_{k} \bigr| ^{2}\dx{x},
	\end{equation}
	where
	\begin{equation}
		A^{(n)}=\lim_{k\to\infty} 2^{-j_{k}^{(n)}}A(2^{j_{k}^{(n)}}\cdot+y_{k}^{(n)}).
	\end{equation}
\end{theorem}
Moreover  under conditions of Theorem~\ref{thm:PDsob} one has the following "iterated Brezis-Lieb lemma" (\cite[Theorem 4.7.1]{ccbook}):
	\begin{equation}
		\label{eq:2*} 
		 \int_{\R ^{N}} \bigl|u_{k} \bigr| ^{2^*}\dx{x} 
\longrightarrow		\sum_{n\in \N }\int_{\R ^{N}}
		\bigl|w^{(n)} \bigr| ^{2^*}
		\dx{x} .
\end{equation}
An elementary modification of the argument from \cite[Theorem 4.7.1]{ccbook} also gives
	\begin{equation}
	\label{eq:2} 
	\int_{\R ^{N}} V(x)\bigl|u_{k} \bigr| ^{2}\dx{x} 
	\longrightarrow		\sum_{n\in \N }\int_{\R ^{N}}V^{(n)}(x)
	\bigl|w^{(n)} \bigr| ^{2}
	\dx{x}, 
\end{equation}
where
	\begin{equation}
	V^{(n)}=\lim_{k\to\infty}2^{-2j_{k}^{(n)}}V(2^{j_{k}^{(n)}}\cdot+y_{k}^{(n)}).
\end{equation}
\bibliographystyle{amsplain}

\end{document}